\documentclass[11pt]{amsart} \usepackage{hyphenat} \usepackage[T1]{fontenc} \usepackage[utf8]{inputenc} \usepackage{lmodern} \usepackage{microtype} \usepackage{amsmath,amssymb,amsthm,mathtools} \usepackage{tikz-cd} \usepackage{enumitem} \usepackage{hyperref} \usepackage[nameinlink,capitalise,noabbrev]{cleveref} \hypersetup{ colorlinks=true, linkcolor=blue, citecolor=blue, urlcolor=blue } \newtheorem{theorem}{Theorem}[section] \newtheorem{lemma}[theorem]{Lemma} \newtheorem{corollary}[theorem]{Corollary}  \theoremstyle{definition}  \newtheorem{example}[theorem]{Example} \theoremstyle{remark}  \crefname{theorem}{Theorem}{Theorems} \Crefname{theorem}{Theorem}{Theorems} \crefname{lemma}{Lemma}{Lemmas} \Crefname{lemma}{Lemma}{Lemmas} \crefname{corollary}{Corollary}{Corollaries} \Crefname{corollary}{Corollary}{Corollaries} \crefname{proposition}{Proposition}{Propositions} \Crefname{proposition}{Proposition}{Propositions} \crefname{definition}{Definition}{Definitions} \Crefname{definition}{Definition}{Definitions} \crefname{example}{Example}{Examples} \Crefname{example}{Example}{Examples} \crefname{remark}{Remark}{Remarks} \Crefname{remark}{Remark}{Remarks} \newcommand{\kk}{\Bbbk} \newcommand{\C}{\mathcal C} \newcommand{\A}{\mathcal A} \newcommand{\B}{\mathcal B}  \newcommand{\K}{\mathsf K} \newcommand{\Db}{\mathsf D^{b}} \newcommand{\Dsg}{\mathsf D_{\mathrm{sg}}} \newcommand{\Hom}{\operatorname{Hom}} \newcommand{\Ker}{\operatorname{Ker}} \newcommand{\Coker}{\operatorname{Coker}} \renewcommand{\Im}{\operatorname{Im}} \newcommand{\soc}{\operatorname{soc}} \newcommand{\supp}{\operatorname{supp}} \newcommand{\rep}{\operatorname{rep}} \newcommand{\inj}{\operatorname{inj}} \newcommand{\proj}{\operatorname{proj}} \newcommand{\modu}{\operatorname{mod}} \newcommand{\rad}{\operatorname{rad}} \newcommand{\op}{\mathrm{op}} 

\title[Singularity Categories]{Singularity Categories of Locally Bounded Categories with Radical Square Zero}
\author{Ales M. Bouhada}
\address{D\'epartement de math\'ematiques, Universit\'e de Sherbrooke, Sherbrooke, Qu\'ebec, Canada}
\email{alesm.bouhada@gmail.com}

\begin{document}

\begin{abstract}
We study several singularity categories associated with a locally bounded $\kk$-linear category $\C$ whose radical has square zero. Building on the work of Bautista and Liu \cite{BL2017}, we give explicit descriptions of
\[
\Dsg^{b}(\C),\qquad \Dsg^{b}(\C^{\op}),\qquad
\Dsg^{-}(\proj\text{-}\C),\qquad
\Dsg^{+}(\inj\text{-}\C)
\]
as orbit categories of bounded derived categories of suitable semisimple abelian categories of quiver representations. We conclude with examples illustrating how generators of $\Dsg^{b}(\C)$ can be read directly from the quiver of $\C$.
\end{abstract}

\maketitle

\section*{Introduction}
More recently, Bautista and Liu obtained a complete description of the bounded derived category of a connected locally bounded $\kk$-category by means of Galois coverings \cite{BL2017}. Their construction extends the classical covering theory of Bongartz and Gabriel \cite{BG1982} and was further generalized, independently, by Asashiba \cite{Asashiba1997,Asashiba2011} and by Asashiba--Hafezi--Vahed \cite{AHV2018} to derived and singularity categories. The aim of this paper is to continue this program by studying singularity categories of a connected elementary locally bounded $\kk$-category $\C$, where $\kk$ is a field and $\rad^{2}\C=0$. By Gabriel's theorem, such a category may be identified with \[ \C=\kk Q/(\kk Q^{+})^{2}, \] where $Q$ is a connected locally finite quiver. Chen \cite{Chen2011} proved that, for an Artin algebra with radical square zero, the singularity category is triangle equivalent to the category of finitely generated projective modules over a certain algebra which is itself a triangulated abelian category. We obtain an analogous phenomenon for locally finite quivers, using a different method which applies uniformly to both finite- and infinite-dimensional settings. We study the four singularity categories \[ \Dsg^{b}(\C),\qquad \Dsg^{b}(\C^{\op}),\qquad \Dsg^{-}(\proj\text{-}\C),\qquad \Dsg^{+}(\inj\text{-}\C). \] For a finite-dimensional algebra with radical square zero, the categories $\Dsg^{b}(\C)$ and $\Dsg^{b}(\C^{\op})$ need not be equivalent, although they are equivalent for Nakayama algebras. Our objective is to identify a natural triangulated category admitting a Galois covering functor onto $\Dsg^{b}(\C)$. The appropriate source category arises from a Gabriel--Zisman localization of the category of finitely copresented representations; this localization turns out to be semisimple abelian. The main result is the following. \begin{theorem}\label{thm:main-intro} Assume that $\C$ is a locally bounded category with radical square zero. Let $G$ and $H$ be the groups generated by the automorphisms \[ \vartheta=\rho[-r_Q] \qquad\text{and}\qquad \varphi=\rho^{-1}[r_Q], \] respectively. Then there are triangle equivalences \begin{align*} \Dsg^{b}(\C) &\simeq \Db\bigl(\rep^{-,b}(\widetilde Q^{\op})[\Sigma^{-1}]\bigr)/G,\\ \Dsg^{b}(\C^{\op}) &\simeq \Db\bigl(\rep^{+,b}(\widetilde Q^{\op})[\Sigma^{-1}]\bigr)/H,\\ \Dsg^{-}(\proj\text{-}\C) &\simeq \left( \Db\bigl(\rep^{-}(\widetilde Q^{\op})\bigr) \big/ \Db\bigl(\rep^{-,b}(\widetilde Q^{\op})\bigr) \right)/G,\\ \Dsg^{+}(\inj\text{-}\C) &\simeq \left( \Db\bigl(\rep^{+}(\widetilde Q^{\op})\bigr) \big/ \Db\bigl(\rep^{+,b}(\widetilde Q^{\op})\bigr) \right)/H. \end{align*} In particular, $\Dsg^{b}(\C)$ and $\Dsg^{b}(\C^{\op})$ are semisimple abelian categories. \end{theorem} The second and fourth equivalences are dual to the first and third, respectively. The proof of the first equivalence relies on three ingredients: a characterization of finitely copresented representations, the semisimplicity of the corresponding localization, and the construction of a Galois covering functor onto the singularity category. As a consequence, every object of $\Dsg^{b}(\C)$ or $\Dsg^{b}(\C^{\op})$ decomposes as a finite direct sum of shifts of simple modules. This gives an effective method for computing generators directly from the quiver.
\section{Preliminaries and background}

We recall the standard material needed below. For quiver representations and locally bounded categories we refer to \cite{BLP2013,BL2014,BL2017}; for Gabriel--Zisman localization and localization of derived categories we use \cite{GZ1967,Miyachi1991}; and for orbit categories and Galois coverings we use \cite{Keller2005,AHV2018}.

\subsection{Quivers}

A quiver $Q=(Q_0,Q_1,s,t)$ consists of a set $Q_0$ of vertices, a set $Q_1$ of arrows, and source and target maps
\[
s,t\colon Q_1\longrightarrow Q_0.
\]
For an arrow $\alpha\colon x\to y$, its formal inverse $\alpha^{-1}$ has source $y$ and target $x$. A walk is a composable sequence of arrows, formal inverses, and trivial paths. A path is a walk consisting only of arrows. We write $Q_n(x,y)$ for the set of paths of length $n$ from $x$ to $y$, and
\[
x^{+}=\{\alpha\in Q_1\mid s(\alpha)=x\},
\qquad
x^{-}=\{\alpha\in Q_1\mid t(\alpha)=x\}.
\]
The quiver $Q$ is \emph{locally finite} if $x^{+}$ and $x^{-}$ are finite for every $x\in Q_0$. It is \emph{strongly locally finite} if it is locally finite and $Q_n(x,y)$ is finite for all $x,y\in Q_0$ and $n\geq 0$.

For a walk $w$, define its degree by assigning degree $1$ to each arrow, degree $-1$ to each formal inverse, and degree $0$ to each trivial path, and extending additively under concatenation. Thus every path has degree equal to its length. The quiver $Q$ is \emph{gradable} if every closed walk has degree zero. In that case there is a decomposition
\[
Q_0=\coprod_{n\in\mathbb Z}Q_n
\]
such that every arrow goes from $Q_n$ to $Q_{n+1}$. A gradable locally finite quiver is strongly locally finite. Moreover, $Q^{\op}$ is gradable with grading $(Q^{\op})_n=Q_{-n}$.

A morphism of quivers $\psi\colon Q\to Q'$ consists of maps $\psi_0\colon Q_0\to Q'_0$ and $\psi_1\colon Q_1\to Q'_1$ preserving sources and targets. An action of a group $G$ on $Q$ is free if $g\cdot x=x$ implies $g=1$. A morphism $\pi\colon Q\to Q'$ is a Galois $G$-covering if it is surjective on vertices, $\pi\circ g=\pi$ for all $g\in G$, every fibre is a $G$-orbit, and the induced maps
\[
x^{+}\longrightarrow \pi(x)^{+},
\qquad
x^{-}\longrightarrow \pi(x)^{-}
\]
are bijective for every $x\in Q_0$.

\subsection{Path algebras and path categories}

The path algebra $\kk Q$ is the $\kk$-vector space with basis the paths in $Q$, with multiplication given by concatenation whenever defined and zero otherwise. If $Q$ has infinitely many vertices, $\kk Q$ need not be unital; it is therefore often preferable to work with the path category. Its objects are the vertices of $Q$, and its morphism spaces are
\[
\kk Q(x,y)=\kk\text{-span}\{\text{paths from }x\text{ to }y\}.
\]

A family $I=(I(x,y))_{x,y\in Q_0}$ is an admissible ideal if it is a two-sided ideal, if $I(x,y)\subseteq (\kk Q^{+})^{2}(x,y)$, and if, for each vertex $x$, there exists $n\geq 2$ such that
\[
(\kk Q^{+})^{n}(x,-)\subseteq I(x,-),
\qquad
(\kk Q^{+})^{n}(-,x)\subseteq I(-,x).
\]
The quotient category $\kk Q/I$ has the same objects as $\kk Q$ and morphism spaces
\[
(\kk Q/I)(x,y)=\kk Q(x,y)/I(x,y).
\]

A small $\kk$-linear category $\C$ is locally bounded if distinct objects are non-isomorphic, every endomorphism algebra $\C(x,x)$ is local, and
\[
\bigoplus_{y\in Q_0}\C(x,y)\oplus\C(y,x)
\]
is finite-dimensional for every $x$. It is elementary if all simple $\C$-modules are one-dimensional over $\kk$. A bound path category $\kk Q/I$ is locally bounded precisely when $Q$ is locally finite and $I$ is admissible. Hence an elementary locally bounded category with radical square zero can be written as
\[
\C=\kk Q/(\kk Q^{+})^{2}.
\]
Throughout the remainder of the paper, $\C$ denotes such a category.

\subsection{Representations and modules}

A left $\C$-module is a $\kk$-linear covariant functor $\C\to\kk\text{-}\modu$. Its support is
\[
\supp M=\{x\in Q_0\mid M(x)\neq 0\}.
\]
We write $\C\text{-}\modu$ for the category of finitely generated left $\C$-modules. The additive subcategories generated by the finitely generated projectives $\C(x,-)$ and the finitely generated injectives $D\C^{\op}(x,-)$ are denoted by $\C\text{-}\proj$ and $\C\text{-}\inj$, respectively. Every finitely generated module has both a projective cover and an injective envelope.

A representation $M$ of $Q$ consists of finite-dimensional vector spaces $M(x)$ for $x\in Q_0$ and linear maps $M(\alpha)\colon M(x)\to M(y)$ for each arrow $\alpha\colon x\to y$. We denote by $\rep^{b}(Q)$ the category of representations with finite support. The socle of $M$ is the subrepresentation defined by
\[
(\soc M)(x)=\bigcap_{\alpha\in x^{+}}\Ker M(\alpha).
\]
For each $a\in Q_0$, let $S_a$ denote the simple representation concentrated at $a$, and let $I_a$ and $P_a$ denote the corresponding indecomposable injective and projective representations.

A representation $M$ is finitely copresented if it admits an exact sequence
\[
0\longrightarrow M\longrightarrow I\longrightarrow J\longrightarrow 0
\]
with $I,J\in\inj(Q)$. The category of finitely copresented representations is denoted by $\rep^{-,b}(Q)$. Dually, $\rep^{+,b}(Q)$ denotes the category of finitely presented representations. Finally, $\rep^{-}(Q)$ and $\rep^{+}(Q)$ denote the categories of representations whose supports contain no right-infinite path and no left-infinite path, respectively.

\subsection{Localization and singularity categories}

Let $\A$ be an abelian category and $\B\subseteq\A$ a Serre subcategory. Let $\Sigma$ be the class of morphisms $s$ in $\A$ such that $\Ker s$ and $\Coker s$ belong to $\B$. The Gabriel--Zisman localization $\A[\Sigma^{-1}]$ is abelian, and the canonical functor
\[
P_{\Sigma}\colon \A\longrightarrow \A[\Sigma^{-1}]
\]
is exact. Morphisms in the localization are represented by roofs
\[
X\xrightarrow{f}Z\xleftarrow{s}Y,
\qquad s\in\Sigma.
\]

Let $\A$ be an additive subcategory of an abelian category. Following Zhou and Zimmermann, define
\begin{align*}
\Dsg^{-,b}(\A)&=\K^{-,b}(\A)/\K^{b}(\A),\\
\Dsg^{+,b}(\A^{\op})&=\K^{+,b}(\A)/\K^{b}(\A),\\
\Dsg^{-}(\A)&=\K^{-}(\A)/\K^{b}(\A),\\
\Dsg^{+}(\A)&=\K^{+}(\A)/\K^{b}(\A).
\end{align*}
When $\A=\proj\text{-}\B$, one recovers the usual bounded singularity category
\[
\Dsg^{b}(\B)=\Db(\B)/\K^{b}(\proj\text{-}\B).
\]
In particular, $\Dsg^{b}(\C)=0$ if and only if $\C$ has finite global dimension.

\subsection{Galois covering functors}

All categories below are assumed to be skeletally small and $\kk$-linear. Let a group $G$ act on a $\kk$-category $\A$. The action is free if $gX\not\simeq X$ for every non-identity $g\in G$, and locally bounded if
\[
\Hom_{\A}(X,gY)=0
\]
for all but finitely many $g\in G$, for every pair $X,Y\in\A$.

A $\kk$-linear functor $E\colon\A\to\B$ is $G$-stable if it is equipped with natural isomorphisms
\[
\delta_g\colon E\circ g\xrightarrow{\sim}E
\]
satisfying the usual cocycle condition. It is a Galois $G$-precovering if, for all $X,Y\in\A$, the map
\[
\bigoplus_{g\in G}\Hom_{\A}(X,gY)
\longrightarrow
\Hom_{\B}(E(X),E(Y)),
\qquad
(u_g)_{g\in G}\longmapsto\sum_{g\in G}\delta_{g,Y}E(u_g),
\]
is an isomorphism. A dense Galois $G$-precovering is called a Galois $G$-covering.

If $G=\langle F\rangle$ is cyclic, the orbit category $\A/G$ has the same objects as $\A$ and morphism spaces
\[
\Hom_{\A/G}(X,Y)=\bigoplus_{n\in\mathbb Z}\Hom_{\A}(X,F^{n}Y).
\]
Even when $\A$ is triangulated, the orbit category need not inherit a triangulated structure automatically.

\section{Main results}

We assume throughout that representations of $Q^{\op}$ are locally finite with respect to the grading, namely
\[
(Q^{\op})_n\cap\supp M
\quad\text{is finite for every }n\in\mathbb Z.
\]

\begin{lemma}\label{lem:finitely-copresented}
Let $Q$ be a gradable locally finite quiver and let $M\in\rep(Q^{\op})$. The following conditions are equivalent:
\begin{enumerate}[label=\textup{(\arabic*)}]
\item there exists a short exact sequence
\[
0\longrightarrow N\longrightarrow M\longrightarrow L\longrightarrow 0
\]
with $N\in\rep^{b}(Q^{\op})$ and $L\in\inj(Q^{\op})$;
\item $M$ is finitely copresented.
\end{enumerate}
\end{lemma}

\begin{proof}
Let $F$ be the functor introduced by Bautista and Liu. For $M\in\rep(Q^{\op})$, the complex $F(M)$ has components
\[
F(M)^{n}=\bigoplus_{x\in Q_{-n}}\C(x,-)\otimes_{\kk}M(x)
\]
and differential determined by the matrices
\[
d_{F(M)}^{n}(y,x)
=
\sum_{\alpha\in Q_1(y,x)}
\C(\alpha,-)\otimes M(\alpha^{\op}).
\]
We show that condition~\textup{(1)} is equivalent to boundedness of the cohomology of $F(M)$, and that the latter is equivalent to finite copresentation.

Assume first that $F(M)$ has bounded cohomology. Since $F(M)$ is bounded above, the support of $M$ contains no right-infinite path. Hence $\soc M$ is essential. Let $\ell$ be the largest integer for which $H^{\ell}(F(M))\neq 0$. A direct inspection of the differential shows that $\soc M$ vanishes in all sufficiently low degrees. Consequently, $\soc M$ has finite support, and $M$ is finitely cogenerated. Let $M\hookrightarrow J$ be an injective envelope. Truncating $J$ along the grading produces an injective representation $L$ and a finite-support subrepresentation $N$ fitting into an exact sequence
\[
0\longrightarrow N\longrightarrow M\longrightarrow L\longrightarrow 0.
\]
This proves condition~\textup{(1)}.

Conversely, suppose that such a short exact sequence exists. Applying $F$ gives a long exact cohomology sequence. Since $N$ has finite support and $L$ is injective, both $F(N)$ and $F(L)$ have bounded cohomology; therefore so does $F(M)$.

Finally, if $M$ is finitely copresented, there is an exact sequence
\[
0\longrightarrow M\longrightarrow I\longrightarrow J\longrightarrow 0
\]
with $I,J$ injective. Since $F(I)$ and $F(J)$ have bounded cohomology, the same is true of $F(M)$. The preceding argument then yields condition~\textup{(1)}.
\end{proof}

Let
\[
\Sigma=
\left\{
 s\in\rep^{-,b}(Q^{\op})
 \;
 \middle|
 \;
 \Ker s,\Coker s\in\rep^{b}(Q^{\op})
\right\}.
\]

\begin{lemma}\label{lem:semisimple-localization}
Let $Q$ be a gradable locally finite quiver. Then
\[
\rep^{-,b}(Q^{\op})[\Sigma^{-1}]
\]
is a semisimple abelian category.
\end{lemma}

\begin{proof}
The subcategory $\rep^{b}(Q^{\op})$ is a Serre subcategory of $\rep^{-,b}(Q^{\op})$, so the localization is abelian. By lemma 2.1, every object of $\rep^{-,b}(Q^{\op})$ becomes isomorphic in the localization to a finite direct sum of injective representations.

Consider a short exact sequence
\[
0\longrightarrow A\xrightarrow{f}B\xrightarrow{g}C\longrightarrow 0
\]
in the localized category. We may assume that $A$, $B$, and $C$ are represented by injective objects. Represent $f$ by a roof
\[
A\xrightarrow{\alpha}Z\xleftarrow{\beta}B,
\]
where $\Ker\beta$, $\Coker\beta$, and $\Ker\alpha$ have finite support. Since $\rep^{-,b}(Q^{\op})$ is hereditary and $A$ is injective, $\Im\alpha$ is injective. Hence
\[
0\longrightarrow \Im\alpha\xrightarrow{i}Z\xrightarrow{p}\Coker\alpha\longrightarrow 0
\]
splits. In the localization, this sequence is isomorphic to the original short exact sequence, as expressed by the commutative diagram
\[
\begin{tikzcd}
0 \arrow[r] & \Im\alpha \arrow[r,"i"] \arrow[d,"\sim"']
& Z \arrow[r,"p"] \arrow[d,"\beta^{-1}"]
& \Coker\alpha \arrow[r] \arrow[d,"\sim"] & 0\\
0 \arrow[r] & A \arrow[r,"f"'] & B \arrow[r,"g"'] & C \arrow[r] & 0.
\end{tikzcd}
\]
Therefore every short exact sequence splits, and the localization is semisimple.

\end{proof}

Let $r_Q$ be the grading period of $Q$, namely the least positive degree of a closed walk; by convention, $r_Q=0$ when $Q$ is gradable. Let $\widetilde Q$ be a gradable connected component of the repetition quiver $Q^{\mathbb Z}$. The translation
\[
\rho(x,n)=(x,n+r_Q)
\]
generates a torsion-free group, and the canonical map
\[
\pi\colon\widetilde Q\longrightarrow Q,
\qquad
\pi(x,n)=x,
\]
is the minimal gradable covering. Set
\[
\widetilde\C=\kk\widetilde Q^{\op}/(\kk\widetilde Q^{+})^{2}.
\]
The group generated by $\vartheta=\rho[-r_Q]$ acts naturally on
\[
\Db\bigl(\rep^{-,b}(\widetilde Q^{\op})[\Sigma^{-1}]\bigr).
\]

\begin{theorem}\label{lem:galois-covering}
Let $Q$ be a locally finite quiver. There exists a Galois covering functor
\[
\Db\bigl(\rep^{-,b}(\widetilde Q^{\op})[\Sigma^{-1}]\bigr)
\longrightarrow
\Dsg^{b}(\C).
\]
If $Q$ is gradable, this functor is a triangle equivalence.
\end{theorem}

\begin{proof}
We divide the proof into three steps.

\smallskip
\noindent\emph{Step 1: the gradable case.}
Let
\[
\Db_{\rep^{b}(\widetilde Q^{\op})}
\bigl(\rep^{-,b}(\widetilde Q^{\op})\bigr)
\]
denote the full subcategory of complexes whose cohomology objects have finite support. Since $\rep^{-}(\widetilde Q^{\op})$ is hereditary, the natural functor
\[
\Db\bigl(\rep^{b}(\widetilde Q^{\op})\bigr)
\longrightarrow
\Db_{\rep^{b}(\widetilde Q^{\op})}
\bigl(\rep^{-,b}(\widetilde Q^{\op})\bigr)
\]
is an equivalence. The derived equivalence of Bautista and Liu, followed by the projection to the singularity category, annihilates this thick subcategory. Hence it factors through a functor
\[
G\colon
\Db\bigl(\rep^{-,b}(\widetilde Q^{\op})\bigr)
\big/
\Db\bigl(\rep^{b}(\widetilde Q^{\op})\bigr)
\longrightarrow
\Dsg^{b}(\widetilde\C).
\]
Using the calculus of fractions and the derived equivalence, one verifies that $G$ is full, faithful, and dense. By Miyachi's localization theorem \cite{Miyachi1991},
\[
\Db\bigl(\rep^{-,b}(\widetilde Q^{\op})\bigr)
\big/
\Db\bigl(\rep^{b}(\widetilde Q^{\op})\bigr)
\simeq
\Db\bigl(\rep^{-,b}(\widetilde Q^{\op})[\Sigma^{-1}]\bigr).
\]
Thus
\[
\Db\bigl(\rep^{-,b}(\widetilde Q^{\op})[\Sigma^{-1}]\bigr)
\simeq
\Dsg^{b}(\widetilde\C).
\]

\smallskip
\noindent\emph{Step 2: passage from $\widetilde\C$ to $\C$.}
The minimal gradable covering induces a commutative square
\[
\begin{tikzcd}
\Db(\widetilde\C) \arrow[r] \arrow[d,"\pi^{D}"']
& \Dsg^{b}(\widetilde\C) \arrow[d,"\pi^{S}"]\\
\Db(\C) \arrow[r]
& \Dsg^{b}(\C).
\end{tikzcd}
\]
The functor $\pi^{D}$ is a Galois covering. The induced functor $\pi^{S}$ is dense and $G$-stable, and it satisfies
\[
\bigoplus_{k\in\mathbb Z}
\Hom_{\Dsg^{b}(\widetilde\C)}(X,\rho^{k}Y)
\xrightarrow{\sim}
\Hom_{\Dsg^{b}(\C)}(\pi^{S}X,\pi^{S}Y).
\]
Hence $\pi^{S}$ is a Galois covering.

\smallskip
\noindent\emph{Step 3: composition.}
The equivalence of Step~1 intertwines the action of $\vartheta=\rho[-r_Q]$ with the action induced by $\rho$ on $\Dsg^{b}(\widetilde\C)$, up to the standard sign twist on complexes. Consequently, the composite
\[
\Db\bigl(\rep^{-,b}(\widetilde Q^{\op})[\Sigma^{-1}]\bigr)
\longrightarrow
\Dsg^{b}(\widetilde\C)
\longrightarrow
\Dsg^{b}(\C)
\]
is a Galois covering. If $Q$ is gradable, then $r_Q=0$ and the covering group is trivial, so the functor is an equivalence.
\end{proof}

\begin{theorem}\label{thm:main}
Assume that $\C$ is locally bounded and $\rad^{2}\C=0$. Let $G=\langle\vartheta\rangle$ and $H=\langle\varphi\rangle$, where
\[
\vartheta=\rho[-r_Q],
\qquad
\varphi=\rho^{-1}[r_Q].
\]
Then there are triangle equivalences
\begin{align*}
\Dsg^{b}(\C)
&\simeq \Db\bigl(\rep^{-,b}(\widetilde Q^{\op})[\Sigma^{-1}]\bigr)/G,\\
\Dsg^{b}(\C^{\op})
&\simeq \Db\bigl(\rep^{+,b}(\widetilde Q^{\op})[\Sigma^{-1}]\bigr)/H,\\
\Dsg^{-}(\proj\text{-}\C)
&\simeq
\left(
\Db\bigl(\rep^{-}(\widetilde Q^{\op})\bigr)
\big/
\Db\bigl(\rep^{-,b}(\widetilde Q^{\op})\bigr)
\right)/G,\\
\Dsg^{+}(\inj\text{-}\C)
&\simeq
\left(
\Db\bigl(\rep^{+}(\widetilde Q^{\op})\bigr)
\big/
\Db\bigl(\rep^{+,b}(\widetilde Q^{\op})\bigr)
\right)/H.
\end{align*}
In particular, $\Dsg^{b}(\C)$ and $\Dsg^{b}(\C^{\op})$ are semisimple abelian categories.
\end{theorem}

\begin{proof}
The first equivalence follows from \cref{lem:galois-covering}. The orbit category inherits a translation functor, and we transport the distinguished triangles from $\Dsg^{b}(\C)$ across the equivalence of additive categories. This gives the orbit category a triangulated structure for which the equivalence is exact. Since the localized representation category is semisimple abelian, so is the resulting orbit category.

The third equivalence is obtained by extending the Bautista--Liu equivalence to the bounded-above derived category and repeating the same localization and covering argument. The second and fourth equivalences follow by duality.
\end{proof}

\begin{corollary}\label{cor:decomposition}
The following statements hold.
\begin{enumerate}[label=\textup{(\arabic*)}]
\item Every object of $\Dsg^{b}(\C)$ or $\Dsg^{b}(\C^{\op})$ is a finite direct sum of shifts of semisimple modules.
\item For every vertex $a\in Q_0$, the simple module $S_a$ satisfies
\[
S_a\simeq
\bigoplus_{\alpha\in a^{+}}S_{t(\alpha)}[1]
\qquad\text{in }\Dsg^{b}(\C).
\]
\item Dually,
\[
S_b\simeq
\bigoplus_{\alpha\in b^{-}}S_{s(\alpha)}[1]
\qquad\text{in }\Dsg^{b}(\C^{\op}).
\]
\end{enumerate}
\end{corollary}

\begin{proof}
Let $Z\in\Dsg^{b}(\C)$. By the Galois covering, $Z$ lifts to an object of
\[
\Db\bigl(\rep^{-,b}(\widetilde Q^{\op})[\Sigma^{-1}]\bigr).
\]
Since the localized category is semisimple, the lift is a finite direct sum of shifts of injective representations. Under the equivalence with the singularity category of $\widetilde\C$, these injectives correspond to shifts of simple modules. The covering functor preserves simples, proving~\textup{(1)}.

For~\textup{(2)}, choose $b\in\widetilde Q_0$ with $\pi(b)=a$. In $\rep^{-}(\widetilde Q^{\op})$ there is an exact sequence
\[
0\longrightarrow S_b\longrightarrow I_b
\longrightarrow
\bigoplus_{\beta\in b^{+}}I_{t(\beta)}
\longrightarrow 0.
\]
The finite-support term becomes zero after localization, and the claimed relation follows after applying the equivalence and descending along the covering. Statement~\textup{(3)} is dual.
\end{proof}

\section{Examples}

The preceding corollary gives a direct procedure for computing generators of $\Dsg^{b}(\C)$ from the quiver.

\begin{example}
Suppose that a portion of $Q$ contains arrows from a vertex $3$ to vertices $1$ and $2$. Choose lifts $a_1,a_2,a_3\in\widetilde Q_0$ with $\pi(a_i)=i$. In $\widetilde Q^{\op}$ there is an exact sequence
\[
0\longrightarrow S_{a_3}
\longrightarrow I_{a_3}
\longrightarrow I_{a_1}\oplus I_{a_2}
\longrightarrow 0.
\]
Since $S_{a_3}$ has finite support, it vanishes in the localization. Hence
\[
I_{a_3}\simeq I_{a_1}\oplus I_{a_2},
\]
and therefore
\[
S_3\simeq S_1[1]\oplus S_2[1]
\qquad\text{in }\Dsg^{b}(\C).
\]
Applying the same argument successively to the remaining vertices expresses every simple module in terms of a smaller generating family. In the quiver considered in the original example, all objects are finite direct sums of shifts of $S_{12}$ tensored with finite-dimensional $\kk$-vector spaces.
\end{example}

\begin{example}
Let $Q=A_{\infty}^{\infty}$ be the doubly infinite linearly oriented quiver
\[
\cdots\longrightarrow \bullet\longrightarrow \bullet\longrightarrow \bullet\longrightarrow\cdots.
\]
Then every object of $\Dsg^{b}(\C)$ is of the form
\[
\bigoplus_i S_a[t_i]\otimes_{\kk}V_i
\]
for some vertex $a$, integers $t_i$, and finite-dimensional vector spaces $V_i$.
\end{example}

\begin{example}
Let $Q$ be a finite quiver containing an oriented cycle. Its minimal gradable covering is an infinite path. Consequently, the preceding example applies, and every object of $\Dsg^{b}(\C)$ is a finite direct sum of shifts of one simple module, tensored with finite-dimensional vector spaces.
\end{example}

The vector spaces occurring in these decompositions are generated by suitable families of paths, as in the proof of lemma 2.1.

\end{document}